\documentclass[graybox]{svmult}

\usepackage{type1cm}        
\usepackage{makeidx}         
\usepackage{graphicx}        
\usepackage{multicol}        
\usepackage[bottom]{footmisc}

\usepackage{newtxtext}       %
\usepackage[varvw]{newtxmath}       

\makeindex             

\usepackage{tikz}
\usepackage{tkz-euclide}
\usepackage{hyperref}

\begin{document}

\title*{A survey on the resolvent convergence}
\author{Joaquim Duran\orcidID{0009-0002-0586-7589}}
\institute{Joaquim Duran \at Centre de Recerca Matem\`atica, Edifici C, Campus Bellaterra, 08193 Bellaterra, Spain, \email{jduran@crm.cat}}

\maketitle

\abstract*{This chapter deals with the notion of the resolvent of a self-adjoint operator. We pay special attention to the convergence of unbounded self-adjoint operators in several resolvent senses, and how they are related to the convergence of their spectra. We also explore the relations that these notions of convergence have with the so-called strong graph limit, $G$-convergence, and $\Gamma$-convergence.}

\abstract{This chapter deals with the notion of the resolvent of a self-adjoint operator. We pay special attention to the convergence of unbounded self-adjoint operators in several resolvent senses, and how they are related to the convergence of their spectra. We also explore the relations that these notions of convergence have with the so-called strong graph limit, $G$-convergence, and $\Gamma$-convergence.}

\section{Motivation for a complex setting} \label{sec:Motivation}

The \emph{resolvent} is a suitable notion for studying the spectra and the convergence of densely defined, linear, unbounded, self-adjoint transformations on infinite dimensional, separable Hilbert spaces. Throughout this chapter we will only consider separable Hilbert spaces over $\mathbb C$, which shall be denoted by $(H, \langle \cdot, \cdot \rangle_H, \|\cdot\|_H)$ ---or, when no confusion arises, by $H$---, unless otherwise specified. 

\begin{definition}
    An \emph{operator} $T$ on a (complex and separable) Hilbert space~$H$ is a linear transformation from its domain, a linear and dense subspace of $H$ denoted by $\mathrm{Dom}(T)$, into $H$.
\end{definition}

Working in the complex setting is not a loss of generality, because every real Hilbert space can be embedded ---in the sense that we explain afterward--- into a complex Hilbert space ---called its \emph{complexification}---, and every real operator can be uniquely extended to a complex operator preserving the embedding ---also called its \emph{complexification}. We point out how.

On the one hand, if $(H, \langle \cdot, \cdot \rangle_H, \|\cdot\|_H)$ is a real Hilbert space, then it can be extended to be a complex Hilbert space in the following sense: there exists a Hilbert space $(H^\mathbb C, \langle \cdot, \cdot \rangle_{H^\mathbb C}, \|\cdot\|_{H^\mathbb C})$ over $\mathbb C$ and a map $\Upsilon : H \to H^\mathbb C$ such that:
\begin{enumerate}
    \item \label{item:inclusion} $\Upsilon$ is linear,
    \item \label{item:inner} $\langle \Upsilon u, \Upsilon v \rangle_{H^\mathbb C} = \langle u, v \rangle_{H}$ for all $u,v\in H$, and
    \item \label{item:decomposition} for every $h\in H^\mathbb C$ there exist unique $u,v\in H$ such that $h=\Upsilon u + i \Upsilon v$.
\end{enumerate}
Actually, $H^\mathbb C$ is (up to bijective isometries) the Cartesian product $H^\mathbb C := H\times H$ endowed with the operations
\begin{equation*}
    (u_1,v_1)+(u_2,v_2) = (u_1+u_2,v_1+v_2) \quad \text{and} \quad (\alpha+i\beta)(u,v) = (\alpha u - \beta v, \alpha v + \beta u)
\end{equation*}
for $u_1,u_2,v_1,v_2\in H$ and $\alpha,\beta\in \mathbb R$.\footnote{More generally, this is the definition of the complexification of a real vector space $H$.} The scalar product is defined by
\begin{equation*}
    \langle (u_1,v_1), (u_2,v_2) \rangle_{H^\mathbb C} :=  \langle u_1, u_2 \rangle_{H} + \langle v_1, v_2 \rangle_{H} + i \langle v_1, u_2 \rangle_{H} - i \langle u_1, v_2 \rangle_{H},
\end{equation*}
the associated norm is defined by
\begin{equation*}
    \|(u,v)\|_{H^\mathbb C} := \left( \|u\|_H^2+\|v\|_H^2 \right)^{1/2}
\end{equation*}
and the map $\Upsilon$ is the so-called \emph{standard embedding} of $H$ into $H^\mathbb C$, defined by 
\begin{equation*}
    \begin{matrix}
        \Upsilon: & H & \to & H^\mathbb C \\
        & u & \mapsto & (u,0).
    \end{matrix}
\end{equation*}
With an abuse of notation, the complexification of $H$ is sometimes written as $H^\mathbb C \equiv H\oplus H$, under the identification $(u,v) \equiv u+iv$.

On the other hand, if $T:\mathrm{Dom}(T)\to H$ is an $\mathbb R$-linear operator in a real Hilbert space~$H$, then it can be uniquely extended (up to bijective isometries) into a $\mathbb C$-linear operator $T^\mathbb C: \mathrm{Dom}(T)^\mathbb C \to H^\mathbb C$ such that $T^\mathbb C\Upsilon u = \Upsilon Tu$ for every $u\in H$. Actually, $T^\mathbb C$ is the $\mathbb C$-linear operator defined by 
\begin{equation*}
    \begin{matrix}
        T^\mathbb C: & \mathrm{Dom}(T)^\mathbb C &\to &  H^\mathbb C \\
        & (u,v) & \mapsto & (Tu,Tv).
    \end{matrix} 
\end{equation*}
It is clear from the definition that if $\mathrm{Dom}(T)$ is $\mathbb R$-linear and dense in $H$, then $\mathrm{Dom}(T)^\mathbb C$ is $\mathbb C$-linear and dense in $H^\mathbb C$; hence, $T^\mathbb C$ is indeed an operator. In addition, it is straightforward to verify that:
\begin{enumerate}
    \item $T^\mathbb C$ is self-adjoint in $H^\mathbb C$ provided that $T$ is self-adjoint in $H$.
    \item $T^\mathbb C$ is bijective in $H^\mathbb C$ if and only if $T$ is bijective in $H$. In such case, $(T^\mathbb C)^{-1} = (T^{-1})^\mathbb C$.
\end{enumerate}

In Definition~\ref{def:adjoint} we recall the notions of \emph{adjoint} of an operator $T$ and of \emph{self-adjointness}.\footnote{The reader shall not confuse the notation $T^*$ for the adjoint of an operator $T$ with the complex conjugate of a number $z\in \mathbb C$, denoted $\overline z$.} As shall be clear once we define the resolvent of a self-adjoint operator in Section~\ref{sec:definitions}, if one is given real self-adjoint operators defined on a real Hilbert space, by applying the results of this chapter to the complexifications one can inherit properties of the spectra and the convergence of the original real operators; see Remark~\ref{rmk:aboutDef}.\ref{item:rmkReal}.

\begin{definition} \label{def:adjoint}
    Let $T$ be an operator on a Hilbert space $H$. Let $D$ be the set of $u\in H$ for which there exists $u^*\in H$ satisfying that
    \begin{equation*}
        \langle T v, u \rangle_H = \langle v, u^* \rangle_H \quad \text{for all } v \in \mathrm{Dom}(T).
    \end{equation*}
    The \emph{adjoint} of $T$ is the operator whose domain is $\mathrm{Dom}(T^*):=D$, and whose action on each $u\in \mathrm{Dom}(T^*)$ is $T^* u := u^*$. An operator $A$ is called \emph{self-adjoint} if it coincides with its adjoint $A^*$, namely, if $A=A^*$.
\end{definition}

\section{Resolvent and spectrum of an operator} \label{sec:definitions}

In the finite dimensional setting, the spectrum of a self-adjoint operator $A$ is the set of its eigenvalues, that is, the numbers $\lambda\in \mathbb R$ for which the operator $A-\lambda I$ is not injective; $I$ denotes here the identity. In finite dimensions, this is equivalent to saying that the spectrum is the set of $\lambda\in \mathbb R$ for which the operator $A-\lambda I$ is not a bijection.

In the infinite dimensional setting, it is in general no longer true that the set of $\lambda\in \mathbb R$ for which the operator $A-\lambda I$ is not a bijection consists of eigenvalues of $A$, although they all belong to this set. However, it gives a notion for the \emph{spectrum} of $A$. This is what motivates the definition of the \emph{resolvent set} and the \emph{resolvent}.

\begin{definition} \label{def:resolvent}
    Let $A:\mathrm{Dom}(A)\to H$ be a self-adjoint operator. A number $\lambda\in \mathbb C$ is said to be in the \emph{resolvent set} of $A$, denoted $\rho(A)$, if the operator $A - \lambda I:\mathrm{Dom}(A)\to H$ is a bijection with bounded inverse, where $I$ is the identity in $H$. In such case, 
    \begin{equation*}
        R_\lambda(A) := (A - \lambda I)^{-1} 
    \end{equation*}
    is called the \emph{resolvent} of $A$ at $\lambda$. The \emph{spectrum} of $A$ is then defined to be the set $\sigma(A):=\mathbb C\setminus \rho(A)$.
\end{definition}

With an abuse of notation, from now on we will omit the identity $I$ and write $A-\lambda$.

\begin{remark} \label{rmk:aboutDef}
    Several remarks about this definition are in order. 
    \begin{enumerate}
        \item If the inverse of $A - \lambda$ exists, then it is linear, and by the closed graph Theorem~\cite[Theorem III.12]{ReedSimon1980} it is also bounded. Hence, the boundedness requirement for the inverse in the definition is redundant. However, the resolvent set uses to appear in the literature defined like this; see \cite[Definition in pg. 253]{ReedSimon1980} or \cite[Section 2.4]{Teschl2014}.
        \item \label{item:resolventNoEmpty} By the self-adjointness of $A$, it holds that $\mathbb C \setminus \mathbb R \subseteq \rho(A)$ \cite[Theorem VIII.3]{ReedSimon1980} ---and consequently $\sigma(A)\subseteq \mathbb R$---, hence $\rho(A)$ is never empty. Actually, for every $u\in \mathrm{Dom}(A)$ there holds
    \begin{equation*}
        \|(A-\lambda)u\|_H^2 = \langle (A-\lambda)u, (A-\lambda)u \rangle_H = \| (A-\mathrm{Re}\lambda)u \|_H^2+|\mathrm{Im}\lambda|^2 \|u\|_H^2.
    \end{equation*}
    Hence, 
    \begin{equation} \label{eq:BoundResolvent}
        \|(A-\lambda)^{-1}\|_{H\to H} \leq \frac{1}{|\mathrm{Im}\lambda|} \quad \text{ for every } \lambda\in \mathbb C \setminus \mathbb R.
    \end{equation}
    \item \label{item:rmkAdjointResolvent} By the self-adjointness of $A$, the adjoint of the resolvent $R_\lambda(A)$ satisfies $R_\lambda(A)^*=R_{\overline \lambda}(A)=(A-\overline \lambda)^{-1}$. Indeed, since 
    \begin{equation*}
        \langle u,v \rangle_H = \langle (A-\lambda)R_\lambda(A) u,v \rangle_H = \langle R_\lambda(A) u, (A-\overline \lambda)v \rangle_H = \langle u, R_\lambda(A)^*(A-\overline \lambda)v \rangle_H
    \end{equation*}
    holds for every $u\in H$ and every $v\in \mathrm{Dom}(A)$ ---and $A$ is densely defined in $H$---, we must have $R_\lambda(A)^*(A-\overline \lambda) = I$.
    \item The same definition is meaningful for (densely defined) closed operators, not necessarily self-adjoint. The interested reader may look at \cite[Section VIII.1]{ReedSimon1980} and \cite[Section 2.4]{Teschl2014}, where the resolvents are introduced for such operators ---beware that in~\cite{ReedSimon1980} the resolvent of one such operator $T$ at $\lambda\in\rho(T)$ is defined to be $(\lambda I-T)^{-1}$, with opposite sign. However, for technical reasons, self-adjointness has to be assumed as soon as one deals with the convergence of operators through their resolvents. Since this will be the main topic of this chapter, we already pose ourselves into the self-adjoint setting.
    \item \label{item:rmkReal} Let $H$ be a real Hilbert space and let $A:\mathrm{Dom}(A)\to H$ be a real self-adjoint operator. We could analogously define its resolvent set $\rho(A)$ to be the set of $\lambda\in \mathbb R$ for which $A-\lambda:\mathrm{Dom}(A)\to H$ is a bijection with bounded inverse, the resolvent of $A$ at one of such $\lambda$ to be $R_\lambda(A) = (A-\lambda)^{-1}$, and its spectrum to be the set $\sigma(A)=\mathbb R \setminus \rho(A)$; see \cite[pg. 133]{DalMaso1993}. However, if $H^\mathbb C$ and $A^\mathbb C$ denote the complexifications of $H$ and $A$, respectively, it easily follows that: 
        \begin{enumerate}
            \item $A^\mathbb C$ is self-adjoint in $H^\mathbb C$,
            \item $\rho(A^\mathbb C) = \rho(A)\cup (\mathbb C \setminus \mathbb R)$,
            \item $R_\lambda(A^\mathbb C) = (R_\lambda(A))^\mathbb C$ for every $\lambda\in \rho(A)$, and
            \item $\sigma(A^\mathbb C) = \sigma(A)$.
        \end{enumerate}
    Hence, there is no loss of generality if we work with the complexifications.
    \end{enumerate}
\end{remark} 

The spectrum of a self-adjoint operator can be decomposed in several ways into~several subsets, each of them satisfying some properties; see \cite[Sections VI.3 and~VII.2]{ReedSimon1980} and \cite[Section 3.3]{Teschl2014}. In this chapter, we will be interested in the following spectral decomposition.

\begin{definition}
    Let $A$ be a self-adjoint operator. 
    \begin{enumerate}
        \item The \emph{point spectrum} of $A$, denoted $\sigma_{\mathrm{p}}(A)$, is the set of all the eigenvalues of $A$.
        \item The \emph{discrete spectrum} of $A$, denoted $\sigma_{\mathrm{d}}(A)$, is the subset of $\sigma_{\mathrm{p}}(A)$ consisting of eigenvalues which are isolated in $\sigma(A)$ and have finite multiplicity.
        \item The \emph{essential spectrum} of $A$, denoted $\sigma_{\mathrm{ess}}(A)$, is the set $\sigma_{\mathrm{ess}}(A)= \sigma(A)\setminus \sigma_{\mathrm{d}}(A)$.
    \end{enumerate}
    In order to emphasize that the spectrum of $A$ is the disjoint union of its discrete and essential spectra, we write $\sigma(A) = \sigma_{\mathrm{d}}(A)\sqcup\sigma_{\mathrm{ess}}(A)$.
\end{definition}

As a consequence of Remark~\ref{rmk:aboutDef}.\ref{item:resolventNoEmpty}, the resolvent of a self-adjoint operator $A$ always exists for some $\lambda$. Moreover, notice that although $A$ could be unbounded and not defined everywhere in $H$, its resolvent at every $\lambda\in \rho(A)$ is a bounded operator defined everywhere in $H$. In addition, as the following theorem asserts ---a proof of which can be found in \cite[Theorem VIII.2]{ReedSimon1980}---, the resolvent $R_\lambda(A)$ is well-behaved in $\lambda\in\rho(A)$.

\begin{theorem}
    Let $A$ be a self-adjoint operator. The resolvent set $\rho(A)$ is a non-empty open subset of $\mathbb C$ on which the map $\rho(A)\ni \lambda\mapsto R_\lambda(A)$ is analytic. That is, $R_\lambda(A)$ has an absolutely convergent power series expansion around every $\lambda\in \rho(A)$. Moreover, the so-called \emph{first resolvent formula} holds:
    \begin{equation*}
        R_\lambda(A) - R_\mu(A) = (\lambda-\mu) R_\lambda(A) R_\mu(A) \quad \text{for every } \lambda,\mu \in \rho(A).
    \end{equation*}
\end{theorem}

\section{Resolvent convergence} \label{sec:ResolventConvergence}

Given a family of unbounded self-adjoint operators $\{A_n\}_{n\in \mathbb N}$, we want to come up with a notion of convergence to a limiting unbounded self-adjoint operator $A$, which provides information about the convergence of the spectra. 

A main drawback is that the operators $A_n$ and $A$ need not be defined everywhere in $H$, nor do they even need to have the same domain. As a consequence, the classical notions of convergence of operators in the weak topology, in the strong topology, or in the operator norm fail to be well defined.

Since it is natural to think that self-adjoint operators are close if certain bounded functions of them are close, resolvents seem to be a convenient tool to provide a notion of convergence. One advantage is that resolvents are bounded operators defined everywhere, and hence the classical notions of convergence are well defined. Another advantage is that, no matter which the family $\{A_n\}_n$ and the limiting operator $A$ are, by self-adjointness $A_n$ and $A$ have resolvents at every $\lambda\in\mathbb C\setminus\mathbb R$; recall Remark~\ref{rmk:aboutDef}.\ref{item:resolventNoEmpty}. All these considerations motivate the following definition.

\begin{definition}
    Let $A_n$ and $A$ be self-adjoint operators. We say that $A_n$ converges to $A$, as $n\to+\infty$,
    \begin{enumerate}
        \item \emph{in the norm resolvent sense} if $R_\lambda(A_n)$ converges to $R_\lambda(A)$ in the operator norm for every $\lambda\in \mathbb C\setminus\mathbb R$. In such case, we write $A_n \overset{\mathrm{n.r.s.}}{\underset{n\to+\infty}{\longrightarrow}} A$;
        \item \emph{in the strong resolvent sense} if $R_\lambda(A_n)u$ converges to $R_\lambda(A)u$ strongly for every $u\in H$ and every $\lambda\in \mathbb C\setminus\mathbb R$. In such case, we write $A_n \overset{\mathrm{s.r.s.}}{\underset{n\to+\infty}{\longrightarrow}} A$.
    \end{enumerate}
\end{definition}

We shall see in Section~\ref{sec:convergenceSpectra} that these notions of convergence indeed give information about the convergence of the spectra of self-adjoint operators. 

Despite the previous considerations, one could expect these definitions to require convergence for every $\lambda\in \mathbb C \setminus \left( \sigma(A) \cup \bigcup_{n\in \mathbb N} \sigma(A_n) \right)$. The following result, a proof of which can be found in \cite[Corollary 6.32]{Teschl2014}, shows that this is actually irrelevant, because resolvent convergence is enough to be tested in a single $\lambda_\star \in \mathbb C \setminus \left( \sigma(A) \cup \bigcup_{n\in \mathbb N} \sigma(A_n) \right)$.

\begin{theorem} \label{thm:RCenoughInOne}
    Suppose that $A_n$ converges to $A$, as $n\to+\infty$, either in the strong or in the norm resolvent sense for one $\lambda_\star \in \mathbb C \setminus \left( \sigma(A) \cup \underset{n\in \mathbb N}{\bigcup} \sigma(A_n) \right)$, namely, that $R_{\lambda_\star}(A_n)$ converges to $R_{\lambda_\star}(A)$ either in operator norm or strongly. Then the same is true for every $\lambda \in \mathbb C \setminus \left( \sigma(A) \cup \underset{n\in \mathbb N}{\bigcup} \sigma(A_n) \right)$.
\end{theorem}

As an immediate consequence of this theorem, we can see that the resolvent convergence is invariant under adding real multiples of the identity. Indeed, if $A_n$ and $A$ are self-adjoint operators such that $A_n$ converges to $A$ in the norm (strong) resolvent sense, then also for every $\mu\in \mathbb R$ the operators $(A_n+\mu)$ and $(A+\mu)$ are self-adjoint, and $R_i(A_n+\mu) = R_{i-\mu}(A_n)$ converges to $R_i(A+\mu) = R_{i-\mu}(A)$ in the operator norm (strongly). This proves the following result.

\begin{corollary} \label{cor:shiftResolvent}
    Let $\mu\in \mathbb R$. $(A_n+\mu)$ converges to $(A+\mu)$ in the norm (strong) resolvent sense if and only if $A_n$ converges to $A$ in the norm (strong) resolvent sense.
\end{corollary}

We can analogously define the notion of \emph{weak resolvent convergence}. 

\begin{definition}
    Let $A_n$ and $A$ be self-adjoint operators. We say that $A_n$ converges to $A$, as $n\to+\infty$, \emph{in the weak resolvent sense} if $R_\lambda(A_n)u$ converges to $R_\lambda(A)u$ weakly for every $u\in H$ and every $\lambda\in \mathbb C\setminus\mathbb R$. In such case, we write $A_n \overset{\mathrm{w.r.s.}}{\underset{n\to+\infty}{\longrightarrow}} A$.
\end{definition}

It is clear that strong resolvent convergence leads to weak resolvent convergence. The following theorem, a proof of which can be found in \cite[Lemma 6.37]{Teschl2014}, states that these notions are actually equivalent for self-adjoint operators.

\begin{theorem} \label{thm:WRimpliesSR}
    Let $A_n$ and $A$ be self-adjoint operators. If $R_{\lambda_\star}(A_n)$ converges to $R_{\lambda_\star}(A)$ weakly in H for some $\lambda_\star \in \mathbb C \setminus \mathbb R$, then $R_{\lambda_\star}(A_n)$ converges to $R_{\lambda_\star}(A)$ strongly in H. 
\end{theorem}

It is also clear that norm resolvent convergence leads to strong resolvent convergence. However, the converse is not true in general, as the following example shows.

\begin{example} \label{ex:Position}
    Let $H=L^2(\mathbb R)$ be the Hilbert space of square integrable, complex valued functions in $\mathbb R$, and consider the so-called \emph{position operator} $P$, defined by
    \begin{eqnarray*}
        \mathrm{Dom}(P) &:=& \left \{\varphi\in L^2(\mathbb R): \int_\mathbb R x^2|\varphi(x)|^2 \, dx <+\infty \right\}, \\
        P(\varphi(x)) &:=& x\varphi(x) \quad \text{for all } f\in \mathrm{Dom}(P).
    \end{eqnarray*}
    Notice that $P$ is an unbounded operator in $L^2(\mathbb R)$, and that it is densely defined in $L^2(\mathbb R)$, because $C^\infty_c(\mathbb R) \subset \mathrm{Dom}(P)$. Moreover, it is straightforward to verify that $P$ is self-adjoint; see \cite[Example in pg. 68]{Teschl2014}. We show that the operator $A_n:=\frac{1}{n}P$ ---de-fined, for $n\in \mathbb N$, by $\mathrm{Dom}(A_n):=\mathrm{Dom}(P)$ and $A_nf = \frac{1}{n}Pf$ for every $f\in \mathrm{Dom}(P)$---, converges to the zero operator, as $n\to+\infty$, in the strong resolvent sense but not in the norm resolvent sense.

    First, we show the strong resolvent convergence to $0$. By Theorem~\ref{thm:RCenoughInOne}, it is enough to show that
    \begin{equation*}
        \|(A_n-i)^{-1} f - (0-i)^{-1} f \|_{L^2(\mathbb R)} \to 0 \quad \text{as } n\to+\infty, \text{ for every } f\in L^2(\mathbb R).
    \end{equation*}
    To this end, set $\varphi_n := (A_n-i)^{-1}f$. Then $f=(A_n-i)\varphi_n = \left(x/n-i\right)\varphi_n$, and hence 
    \begin{equation} \label{eq:auxExPosition}
        \|(A_n-i)^{-1} f - (0-i)^{-1} f \|_{L^2(\mathbb R)}^2 = \left\|\left(\frac{x}{n}-i\right)^{-1}f -if  \right\|_{L^2(\mathbb R)}^2 = \int_\mathbb R \frac{x^2}{x^2+n^2} |f(x)|^2 \ dx.
    \end{equation}
    Since $\left\{\mathbb R\ni x\mapsto \frac{x^2}{x^2+n^2} |f(x)|^2\right\}_n$ is a sequence of measurable functions converging pointwise to $0$ as $n\to+\infty$, which are all dominated by the integrable function $|f(x)|^2$, the Dominated convergence theorem \cite[Theorem A.24]{Teschl2014} ensures that the integral in \eqref{eq:auxExPosition} converges to $0$ as $n\to+\infty$, as desired.

    Next, we show that $A_n$ does not converge to $0$ in the norm resolvent sense as $n\to+\infty$. To this end, for every $k\in \mathbb N$ define $f_k(x):\mathbb R \to \mathbb C$ by
    \begin{equation*}
        f_k(x):= \begin{cases}
            x^{-2} & \text{if } x>k, \\
            0 & \text{elsewhere}.
        \end{cases}
    \end{equation*}
    It is clear that $f_k\in L^2(\mathbb R)\setminus\{0\}$, and by \eqref{eq:auxExPosition} we have that
    \begin{equation*}
        \frac{\|(A_n-i)^{-1} f_k - (0-i)^{-1} f_k \|_{L^2(\mathbb R)}^2}{\|f_k\|_{L^2(\mathbb R)}^2} = \frac{\int_k^{+\infty} \frac{x^2}{x^2+n^2}|x^{-2}|^2 \, dx}{\int_k^{+\infty}|x^{-2}|^2 \, dx} \geq \frac{k^2}{k^2+n^2} \ \ \text{for all } k,n\in \mathbb N.
    \end{equation*}
    As a consequence, 
    \begin{eqnarray*}
        \|(A_n-i)^{-1} - (0-i)^{-1} \|_{L^2(\mathbb R)\to L^2(\mathbb R)}^2 &=& \underset{f\in L^2(\mathbb R)\setminus\{0\}}{\sup} \frac{\|(A_n-i)^{-1} f - (0-i)^{-1} f \|_{L^2(\mathbb R)}^2}{\|f\|_{L^2(\mathbb R)}^2} \\
        &\geq& \underset{k\in \mathbb N}{\sup} \, \frac{k^2}{k^2+n^2} = 1 \quad \text{for all } n\in \mathbb N.
    \end{eqnarray*}
    This prevents $A_n$ from converging to $0$ in the norm resolvent sense as $n\to+\infty$.
\end{example}

We conclude this section pointing out that, on top of being convenient, these definitions of norm and strong resolvent convergence are good generalizations of the classical convergences in the operator norm and strong topology for bounded operators, respectively.

\begin{theorem} \label{thm:GenBounded}
    Let $A_n$ and $A$ be uniformly bounded self-adjoint operators. Then, the following statements hold.
    \begin{enumerate}
        \item $A_n$ converges to $A$ in the norm resolvent sense if and only if $A_n$ converges to $A$ in operator norm.
        \item $A_n$ converges to $A$ in the strong resolvent sense if and only if $A_n$ converges to $A$ strongly.
    \end{enumerate}
\end{theorem}

As a consequence of this theorem, which is a combination of \cite[Theorem VIII.18 and Problem VIII.28]{ReedSimon1980}, the analog for weak resolvent convergence is not true. To see this, let $A_n$ and $A$ be uniformly bounded self-adjoint operators such that $A_n$ converges weakly but not strongly to $A$, and assume in addition that there was convergence in the weak resolvent sense. Then, by Theorem~\ref{thm:WRimpliesSR}, $A_n$ would converge to $A$ in the strong resolvent sense, and hence, by Theorem~\ref{thm:GenBounded}, $A_n$ would converge to $A$ strongly, leading to a contradiction. We summarize this obstruction in the following corollary.

\begin{corollary} \label{cor:BoundedWeak}
    Let $A_n$ and $A$ be uniformly bounded self-adjoint operators. If $A_n$ converges to $A$ weakly but not strongly, then $A_n$ does not converge to $A$ in the weak resolvent sense.
\end{corollary}

\section{Convergence of spectra} \label{sec:convergenceSpectra}

It is straightforward from Definition~\ref{def:resolvent} that the resolvents of self-adjoint operators encode information about their spectra. In this section, we show how the resolvent convergence actually gives information about the convergence of spectra.

\begin{theorem}
    Let $A_n$ and $A$ be self-adjoint operators. If $A_n$ converges to $A$ in the strong resolvent sense, then
    \begin{equation*}
        \sigma(A) \subseteq \underset{n\to+\infty}{\lim} \sigma(A_n),
    \end{equation*}
    meaning that for every $\lambda\in \sigma(A)$ there exists a sequence of $\lambda_n\in \sigma(A_n)$ converging to~$\lambda$.
\end{theorem}

This theorem, a proof of which can be found in \cite[Theorem VIII.24(a)]{ReedSimon1980}, says that the spectrum of the limiting operator can not expand if there is strong resolvent convergence. However, in general it can suddenly contract, as the following example shows.

\begin{example}
    We revisit Example~\ref{ex:Position}, where we have shown that, in $H=L^2(\mathbb R)$, the scaled position operator $A_n:=\frac{1}{n}P$ converges to the zero operator in the strong resolvent sense as $n\to+\infty$ (and not in the norm resolvent sense). The limiting operator, $A=0$, has trivial spectrum $\sigma(A)=\{0\}$. Instead, we show that $\sigma(P) = \mathbb R$, and hence $\sigma(A_n) = \mathbb R$ for every $n\in \mathbb N$.

    Let $\lambda\in \mathbb R$. In order to see that $\lambda\in\sigma(P)$, we claim that it is enough to show that there exists a sequence $\{\varphi_k\}_{k\in \mathbb N} \subset \mathrm{Dom}(P)$ with $\|\varphi_k\|_{L^2(\mathbb R)}=1$ satisfying that $\|(P-\lambda)\varphi_k\|_{L^2(\mathbb R)}\to 0$ as $k\to+\infty$; such a sequence is called a \emph{Weyl sequence}, see \cite[Lemma~2.17]{Teschl2014}. Indeed, if such a sequence exists but $\lambda$ did not belong to $\sigma(P)$, since in particular $\lambda$ would be in the resolvent set of $P$ ---that is, the resolvent $R_\lambda(P)$ would be a bounded operator---, we could bound
    \begin{equation*}
        1 = \|\varphi_k\|_{L^2(\mathbb R)} = \|R_\lambda(P)(P-\lambda)\varphi_k\|_{L^2(\mathbb R)} \leq \|R_\lambda(P)\|_{L^2(\mathbb R)\to L^2(\mathbb R)} \|(P-\lambda)\varphi_k\|_{L^2(\mathbb R)},
    \end{equation*}
    with the right hand side converging to zero as $k\to+\infty$, reaching a contradiction.

    We conclude exhibiting a Weyl sequence for the arbitrary $\lambda\in \mathbb R$. For every $k\in \mathbb N$, define $\varphi_k:\mathbb R \to \mathbb C$ as 
    \begin{equation*}
        \varphi_k(x) := \begin{cases}
            \sqrt k & \text{if } x\in \left( \lambda - \frac{1}{2k}, \lambda + \frac{1}{2k} \right), \\ 
            0 & \text{elsewhere}.
        \end{cases}
    \end{equation*}
    It is clear that $\varphi_k\in \mathrm{Dom}(P)$ and that $\|\varphi_k\|_{L^2(\mathbb R)}=1$ for every $k\in \mathbb N$. Moreover,
    \begin{equation*}
        \|(P-\lambda)\varphi_k\|_{L^2(\mathbb R)}^2 = \int_{\lambda - \frac{1}{2k}}^{\lambda + \frac{1}{2k}} (x-\lambda)^2k \, dx = \frac{1}{12k^2} \longrightarrow 0 \quad \text{as } k\to+\infty.
    \end{equation*}
\end{example}

This example shows the possibility that the spectrum of the limiting operator contracts under strong resolvent convergence, by exhibiting some operators that do not, in addition, converge in the norm resolvent sense ---recall Example~\ref{ex:Position}. The following theorem, a proof of which can be found in \cite[Theorem VIII.23(a)]{ReedSimon1980} and \cite[Satz 9.24(a)]{Weidmann2000}, shows that this can not happen if there is convergence in the norm resolvent sense.

\begin{theorem}
    Let $A_n$ and $A$ be self-adjoint operators. If $A_n$ converges to $A$ in the norm resolvent sense, then
    \begin{equation*}
        \sigma(A) = \underset{n\to+\infty}{\lim} \sigma(A_n),
    \end{equation*}
    meaning that:
    \begin{enumerate}
        \item for every $\lambda\in \sigma(A)$ there exists a sequence of $\lambda_n\in \sigma(A_n)$ converging to $\lambda$, and
        \item if $\lambda_n\in \sigma(A_n)$ converge to some $\lambda$, then $\lambda\in \sigma(A)$.
    \end{enumerate}
    The same holds true for the essential and discrete spectra, namely,
    \begin{equation*}
        \sigma_{\mathrm{ess}}(A) = \underset{n\to+\infty}{\lim} \sigma_{\mathrm{ess}}(A_n) \quad \text{and} \quad \sigma_{\mathrm{d}}(A) = \underset{n\to+\infty}{\lim} \sigma_{\mathrm{d}}(A_n).
    \end{equation*}
\end{theorem}

We conclude this section saying that the convergence of the essential spectra under norm resolvent convergence is not proven throughout \cite{ReedSimon1980}, but one can find a statement reminiscent of this in \cite[Problem VIII.49(c)]{ReedSimon1980}. There, it is asserted that if each $A_n$ has purely discrete spectrum in an interval $(a,b)\subset \mathbb R$ ---meaning that $\sigma(A_n)\cap (a,b) = \sigma_{\mathrm{d}}(A_n) \cap (a,b)$--- and $A_n$ converges to $A$ in the norm resolvent sense, then $A$ also has purely discrete spectrum in $(a,b)$.

\section{Relation with strong graph limit} \label{sec:relationGraph}

In Section~\ref{sec:ResolventConvergence} we have motivated the definition of resolvent convergence, thinking that unbounded self-adjoint operators are close if certain bounded functions of them are close. Alternatively, one could think that unbounded self-adjoint operators are close if their graphs are close. This motivates another notion of convergence which, a priori, seems unrelated to resolvents.

\begin{definition}
    Let $A_n$ be self-adjoint operators in a Hilbert space $H$. We say that a pair of elements $(\varphi,\psi)\in H\times H$ is in the \emph{strong graph limit} of $A_n$ if there exist $\varphi_n\in \mathrm{Dom}(A_n)$ such that both $\varphi_n \to \varphi$ and $A_n\varphi_n \to \psi$ strongly in $H$, as $n\to +\infty$. If the set of such pairs is the graph of a self-adjoint operator $A$, then we say that $A$ is the \emph{strong graph limit} of $A_n$, and we write $A = \underset{n\to +\infty}{\mathrm{s.gr.lim}} \, A_n$.
\end{definition}

Nevertheless, this notion turns out to be equivalent to strong resolvent convergence.

\begin{theorem} \label{thm:SRCiffSGL}
    Let $A_n$ and $A$ be self-adjoint operators. Then $A_n$ converges to $A$ in the strong resolvent sense if and only if $A$ is the strong graph limit of $A_n$.
\end{theorem}

A proof of this result can be found in \cite[Theorem VIII.26]{ReedSimon1980}. Similarly, we can define \emph{weak graph limits}.

\begin{definition}
    Let $A_n$ be self-adjoint operators in a Hilbert space $H$. We say that a pair of elements $(\varphi,\psi)\in H\times H$ is in the \emph{weak graph limit} of $A_n$ if there exist $\varphi_n\in \mathrm{Dom}(A_n)$ such that $\varphi_n \to \varphi$ strongly in $H$ and $A_n\varphi_n \to \psi$ weakly in $H$, as $n\to +\infty$. If the set of such pairs is the graph of a self-adjoint operator $A$, then we say that $A$ is the \emph{weak graph limit} of $A_n$, and we write $A = \underset{n\to +\infty}{\mathrm{w.gr.lim}} \, A_n$.
\end{definition}

An analogous statement to Theorem~\ref{thm:SRCiffSGL} can not hold in the weak sense, because of the following result; see \cite[Problem VIII.26]{ReedSimon1980}.

\begin{theorem} \label{thm:BoundedWRCiffWGL}
    Let $A_n$ and $A$ be uniformly bounded self-adjoint operators. Then $A$ is the weak graph limit of $A_n$ if and only if $A_n$ converges to $A$ in the weak topology.
\end{theorem}

Indeed, if an analogous statement to Theorem~\ref{thm:SRCiffSGL} were true for the weak graph limit, then by Theorem~\ref{thm:BoundedWRCiffWGL} we would have that, for uniformly bounded self-adjoint operators, weak resolvent convergence would be equivalent to convergence in the weak topology, which contradicts Corollary~\ref{cor:BoundedWeak}. This proves the following statement.

\begin{corollary}
    Let $A_n$ and $A$ be uniformly bounded self-adjoint operators. If $A$ is the weak graph limit of $A_n$ and $A_n$ does not converge to $A$ strongly, then $A_n$ does not converge to $A$ in the weak resolvent sense.
\end{corollary}

\section{Relation with $G$- and $\Gamma$-convergence} \label{sec:relationGamma}

We conclude this survey on the resolvent convergence exploring its relation with \emph{$G$-} and \emph{$\Gamma$-convergence}. In the literature ---see, for example, \cite{Braides2002,DalMaso1993}---, \emph{$G$-convergence} of positive self-adjoint operators and \emph{$\Gamma$-convergence} of their associated positive quadratic forms is usually formulated in the context of real Hilbert spaces. Nevertheless, many of the proofs developed in the real Hilbert space setting adapt directly to the complex Hilbert space setting. In \cite[Section 2]{Bedoya2014} it is described the modifications needed to generalize the pertinent results in \cite{DalMaso1993} to the complex Hilbert space setting. As defined in \cite[Definition~1]{Bedoya2014}, the notion of $\Gamma$-convergence is the same as in \cite{DalMaso1993}, and states as follows.

\begin{definition} \label{def:GammaConv}
    Let $Q_n:H\to \mathbb R\cup \{\infty\}$, for $n\in\mathbb N$, and $Q_\infty:H\to \mathbb R\cup \{\infty\}$ be functionals. We say that $Q_n$ \emph{$\Gamma$-converges in the strong topology} to $Q_\infty$, as $n\to+\infty$, if
    \begin{enumerate}
        \item for every $u\in H$ and every $u_n$ converging to $u$ strongly in $H$, it holds that $Q_\infty(u) \leq \underset{n\to+\infty}{\liminf} \, Q_n(u_n)$, and
        \item for every $u\in H$ there exists a sequence $u_n$ converging to $u$ strongly in $H$ such that $Q_\infty(u) = \underset{n\to+\infty}{\lim} \, Q_n(u_n)$.
    \end{enumerate}
\end{definition}

\begin{remark}
    Two remarks are in order about Definition~\ref{def:GammaConv}.
     \begin{enumerate}
         \item With an abuse of notation, we alternatively say that $Q_n$ \emph{strongly $\Gamma$-converges} to $Q_\infty$, and we write $Q_n \overset{\mathrm{S \Gamma}}{\underset{n\to+\infty}{\longrightarrow}} Q_\infty$.
         \item There exists a notion of ``weak" $\Gamma$-convergence, the definition of which is not analogous to Definition~\ref{def:GammaConv}; see \cite[equation (1.13) and the paragraph below it]{Braides2002}. However, assuming, in addition, that there exists a constant $c>0$ such that $Q_n(v) \geq c \|v\|_H^2$ for every $v\in H$ and every $n\in\mathbb N \cup \{\infty\}$, if in Definition~\ref{def:GammaConv} we replace the strong convergence of $u_n$ by weak convergence, then we get a characterization of the so-called \emph{$\Gamma$-convergence in the weak topology} of $Q_n$ to $Q_\infty$; see \cite[Proposition 8.10]{DalMaso1993}. With an abuse of notation, we alternatively say that $Q_n$ \emph{weakly $\Gamma$-converges} to $Q_\infty$, and we write $Q_n \overset{\mathrm{W \Gamma}}{\underset{n\to+\infty}{\longrightarrow}} Q_\infty$.
     \end{enumerate}
\end{remark}

Similarly, the notion of $G$-convergence of equi-coercive self-adjoint operators is the same as in \cite[Definition 13.3]{DalMaso1993}.

\begin{definition}
    Given $c>0$, let $A_n$ and $A_\infty$ be self-adjoint operators such that $\langle A_n u, u\rangle_H \geq c\|u\|_H^2$ for every $u\in \mathrm{Dom}(A_n)$ and every $n\in \mathbb N\cup \{\infty\}$. We say that $A_n$ \emph{strongly $G$-converges} to $A_\infty$, as $n\to+\infty$, if for every $f\in H$ the inverse $A_n^{-1}f$ converges to $A_\infty^{-1}f$ strongly in $H$, and we write $A_n \overset{\mathrm{S G}}{\underset{n\to+\infty}{\longrightarrow}} A_\infty$. If instead of $A_n^{-1}f$ converging to $A_\infty^{-1}f$ strongly in $H$ one considers weak convergence in $H$, then we say that $A_n$ \emph{weakly $G$-converges} to $A_\infty$, and we write $A_n \overset{\mathrm{W G}}{\underset{n\to+\infty}{\longrightarrow}} A_\infty$.
\end{definition}

In view of Theorem~\ref{thm:RCenoughInOne}, if $A_n$ and $A_\infty$ are non-negative self-adjoint operators, then $A_n$ converges to $A_\infty$ in the strong resolvent sense if and only if $A_n+\lambda I$ strongly $G$-converges to $A_\infty+\lambda I$, for every $\lambda>0$.

The following theorem relates the resolvent with $G$- and $\Gamma$-convergence. It is the complex version of \cite[Theorem 13.5 and Corollary 13.7]{DalMaso1993}, a proof of which can be found in \cite[Section 2]{Bedoya2014}. In the theorem, the quadratic form associated to a self-adjoint operator $A$ is the form $Q:H\to \mathbb R\cup\{\infty\}$ defined by
\begin{equation*}
    Q(u) := \begin{cases}
        \langle A u, u \rangle_H & \text{if } u\in \mathrm{Dom}(A), \\
        \infty & \text{otherwise}.
    \end{cases}
\end{equation*}

\begin{theorem} \label{thm:GammaGResolvent}
    Let $A_n$ and $A_\infty$ be self-adjoint operators such that their associated quadratic forms satisfy $Q_n(u)\geq c\|u\|_H^2$ for every $u\in H$ and every $n\in \mathbb N\cup \{\infty\}$, for some $c> 0$. Then, $Q_n$ weakly $\Gamma$-converges to $Q_\infty$ if and only if $A_n$ weakly $G$-converges to $A_\infty$. Moreover, the following statements are equivalent.
    \begin{enumerate}
        \item $A_n$ converges in the strong resolvent sense to $A_\infty$.
        \item $A_n$ strongly $G$-converges to $A_\infty$.
        \item $Q_n$ both strongly and weakly $\Gamma$-converges to $Q_\infty$.
    \end{enumerate}
\end{theorem}

We point out that studying the convergence of the quadratic forms associated to self-adjoint operators is of special interest, because by the min-max Theorem \cite[Theorem~4.14]{Teschl2014} the eigenvalues of a self-adjoint operator $A$ which are below its essential spectrum are given by the min-max levels
\begin{equation*}
    \underset{\substack{ V\subset \mathrm{Dom}(A) \\ \mathrm{dim}(V)=k }}{\inf} \, \, \underset{u\in V\setminus\{0\}}{\sup} \, \dfrac{\langle A u, u \rangle_H}{\|u\|_H^2}, \quad \text{where } k \in \mathbb N.
\end{equation*}

\begin{remark}
    Let $\mu>0$. In \cite[Theorem 13.6]{DalMaso1993} it is shown ---in the real Hilbert space setting--- the equivalence of the following conditions, assuming that $A_n$ and $A_\infty$ are self-adjoint operators whose associated quadratic forms satisfy $Q_n(u)\geq 0$ for every $u\in H$ and every $n\in \mathbb N\cup \{\infty\}$:
    \begin{enumerate}
        \item $A_n$ converges in the strong resolvent sense to $A_\infty$.
        \item $(A_n+\mu)$ strongly $G$-converges to $(A_\infty+\mu)$.
        \item $(Q_n+\mu\|\cdot\|_H^2)$ both strongly and weakly $\Gamma$-converges to $(Q_\infty+\mu\|\cdot\|_H^2)$.
    \end{enumerate}
    This also holds in the complex Hilbert space setting. Indeed, by Corollary~\ref{cor:shiftResolvent}, point $1$ is equivalent to $(A_n+\mu)$ converging in the strong resolvent sense to $(A_\infty+\mu)$, and this is in turn equivalent both to points $2$ and $3$, by Theorem~\ref{thm:GammaGResolvent}.
\end{remark}

It is worth noting that, by Theorem~\ref{thm:GammaGResolvent}, the strong resolvent convergence of self-adjoint operators is not equivalent to the strong $\Gamma$-convergence of their quadratic forms. The main reason for this is the fact that, in general, the weak and strong $\Gamma$-limits of a same sequence may differ; see \cite[Example 6.6]{DalMaso1993}. The following result provides a scenario in which the strong resolvent convergence is equivalent to all the notions of convergence that have been presented in this section.

\begin{theorem} \label{thm:WeakStrongGamma}
    Assume that $(h, \langle \cdot, \cdot \rangle_h, \|\cdot\|_h)$ is a compactly embedded Hilbert subspace of $H$. Let $A_n$ and $A_\infty$ be self-adjoint operators in $H$ such that
    \begin{enumerate}
        \item $\mathrm{Dom}(A_n) \subseteq h$ for every $n\in \mathbb N \cup \{\infty\}$, and
        \item their associated quadratic forms are equi-coercive in $h$, namely, they satisfy 
        \begin{equation*}
            Q_n(u)\geq c\|u\|_h^2 \quad \text{for every } u\in h \text{ and every } n\in \mathbb N\cup \{\infty\}, \text{ for some } c> 0. 
        \end{equation*}
    \end{enumerate}
    Then, the following statements are equivalent.
    \begin{enumerate}
        \item $A_n$ converges in the strong resolvent sense to $A_\infty$.
        \item $A_n$ strongly $G$-converges to $A_\infty$.
        \item $Q_n$ strongly $\Gamma$-converges to $Q_\infty$.
        \item $Q_n$ weakly $\Gamma$-converges to $Q_\infty$.
        \item $A_n$ weakly $G$-converges to $A_\infty$.
    \end{enumerate}
\end{theorem}

A proof of this theorem can be found in \cite[Theorem 13.12]{DalMaso1993}, where the Hilbert spaces are assumed to be real. Nevertheless, the real structure is only used in the proof of \cite[Theorem 13.12]{DalMaso1993} to invoke Theorem~\ref{thm:GammaGResolvent}, which also holds in the complex Hilbert space setting.

The scenario described by the hypothesis of Theorem~\ref{thm:WeakStrongGamma} is usually found in the context of uniform elliptic PDEs in divergence form. The reader can find one such concrete example in Section~\ref{sec:RobinLaplacian}, where we study the resolvent convergence in $H=L^2(\Omega)$ of the Robin Laplacian to the Dirichlet Laplacian as the Robin boundary condition degenerates to the Dirichlet boundary condition; see the notation at the beginning of Section~\ref{sec:RobinLaplacian}. Indeed, the quadratic forms associated to the weak formulations of both the Robin and the Dirichlet Laplacians are equi-coercive in the compactly embedded Hilbert subspace $h=H^1(\Omega)$, on which the operators are defined.

The reader can also find in Section~\ref{sec:Graph} a diagram summarizing the results involving the resolvent convergence that we stated in this survey; see Figure~\ref{fig:graphResolvent}.

\section{Diagram summarizing the resolvent convergence} \label{sec:Graph}

\begin{figure}[!ht]
\centering
\resizebox{1\textwidth}{!}{
\begin{tikzpicture}
\tikzstyle{every node}=[font=\normalsize]

\node [font=\normalsize] at (8.75,15.15) {$(A-\lambda)^{-1} = \,$ resolvent of $A$ at $\lambda\in\rho(A), \quad \sigma(A) = \mathbb C \setminus \rho(A)$};
\node [font=\normalsize] at (8.75,14.65) {$\sigma(A) = \sigma_{\mathrm{d}}(A)\sqcup\sigma_{\mathrm{ess}}(A)$};
\node [font=\normalsize] at (7.47,14) {$\sigma_\mathrm{d}(A) = $};
\node [font=\normalsize] at (9.65,14.15) {isolated eigenvalues};
\node [font=\normalsize] at (9.65,13.75) {of finite multiplicity};

\node [font=\normalsize] at (9,11.75) {$A_n\underset{n\to+\infty}{\overset{\mathrm{n.r.s.}}{\longrightarrow}} A_\infty$};
\draw [-{Implies},double] (9,11.25) -- (9,10.5);

\node [font=\normalsize] at (9,10) {$A_n\underset{n\to+\infty}{\overset{\mathrm{s.r.s.}}{\longrightarrow}} A_\infty$};

\draw [{Implies}-{Implies},double] (8.75,9.5) -- (7.5,8.75);
\node [font=\normalsize] at (7.25,8.25) {$A_n\underset{n\to+\infty}{\overset{\mathrm{w.r.s.}}{\longrightarrow}} A_\infty$};
\draw [{Implies}-{Implies},double] (9.25,9.5) -- (10.5,8.75);
\node [font=\normalsize] at (10.75,8.25) {$A_\infty = \underset{n\to +\infty}{\mathrm{s.gr.lim}} \, A_n$};

\draw [-{Implies},double] (7.5,7.75) -- (8.75,7);
\draw [-{Implies},double] (10.5,7.75) -- (9.25,7);
\node [font=\normalsize] at (9,6.5) {$A_\infty = \underset{n\to +\infty}{\mathrm{w.gr.lim}} \, A_n$};

\draw [ dashed] (0.75,13.75) rectangle  (5.7,7.5);
\begin{scope}[shift={(-5.75, 0)}]

    \node [font=\normalsize] at (9,13.25) {If, in addition, $A_n$ are uniformly};
    \node [font=\normalsize] at (9,12.75) {bounded for $n\in \mathbb N \cup\{\infty\}$};

    \node [font=\normalsize] at (9,11.75) {$A_n \underset{n\to+\infty}{\longrightarrow} A_\infty$ in norm};
    \draw [-{Implies},double] (9,11.25) -- (9,10.5);

    \node [font=\normalsize] at (9,10) {$A_n \underset{n\to+\infty}{\longrightarrow} A_\infty$ strongly};
    \draw [-{Implies},double] (9,9.5) -- (9,8.75);
    
    \node [font=\normalsize] at (9,8.25) {$A_n \underset{n\to+\infty}{\longrightarrow} A_\infty$ weakly};
\end{scope}

\draw [{Implies}-{Implies},double] (5.25,11.75) -- (7.75,11.75);
\draw [{Implies}-{Implies},double] (5.25,10) -- (7.75,10);
\draw [{Implies}-{Implies},double] (7.35,6.65) -- (3.25,7.85);
\draw [-{Implies},double] (6.15,8.25) -- (5,8.25);

\draw [ dashed] (12.25,11.35) rectangle  (19.75,6.5);
\begin{scope}[shift={(6, 0)}]

    \node [font=\normalsize] at (10,10.75) {If, in addition, $\langle A_n \cdot, \cdot\rangle \geq c \|\cdot\|^2$ for $\, \begin{matrix}
        n\in \mathbb N\cup \{\infty\} \\[0.05pt]
        c > 0
    \end{matrix}$};

    \node [font=\normalsize] at (8.5,10) {$A_n \underset{n\to+\infty}{\overset{SG}{\longrightarrow}} A_\infty$};

    \draw [{Implies}-{Implies},double] (8.5,9.4) -- (8.5,8.6);

    \node [font=\normalsize] at (8.75,7.5) {$\begin{cases}
        \langle A_n \cdot, \cdot\rangle \overset{\mathrm{S \Gamma}}{\underset{n\to+\infty}{\longrightarrow}} \langle A_\infty\cdot, \cdot\rangle, \text{ and } \\[7pt]
        \langle A_n \cdot, \cdot\rangle \overset{\mathrm{W \Gamma}}{\underset{n\to+\infty}{\longrightarrow}} \langle A_\infty\cdot, \cdot\rangle
    \end{cases}$};

    \draw [{Implies}-{Implies},double] (10,7.1) -- (11.35,7.1);
    \node [font=\normalsize] at (12.5,7.1) {$A_n \underset{n\to+\infty}{\overset{WG}{\longrightarrow}} A_\infty$};
\end{scope}

\draw [{Implies}-{Implies},double] (10.25,10) -- (13.25,10);

\draw [-{Implies},double] (9.25,10.5) -- (12.65,12);
\node [font=\normalsize] at (14.45,12) {$\sigma(A_\infty) \subseteq \underset{n\to+\infty}{\lim} \sigma(A_n)$};

\draw [-{Implies},double] (9.25,12.25) -- (12.5,13.5);
\node [font=\normalsize] at (15,13.5) {$\begin{cases}
\sigma(A_\infty) = \underset{n\to+\infty}{\lim} \, \sigma(A_n), \\
\sigma_{\mathrm{d}}(A_\infty) = \underset{n\to+\infty}{\lim} \sigma_{\mathrm{d}}(A_n), \text{ and}  \\
\sigma_{\mathrm{ess}}(A_\infty) = \underset{n\to+\infty}{\lim} \, \sigma_{\mathrm{ess}}(A_n)
\end{cases}$};
\end{tikzpicture}
}

\caption{Diagram of the relations between the different notions of convergence of densely defined, self-adjoint operators $A_n$, $n\in\mathbb N \cup\{\infty\}$, in a separable, complex Hilbert space $(H,\langle\cdot,\cdot\rangle, \|\cdot\|)$.}
\label{fig:graphResolvent}
\end{figure}
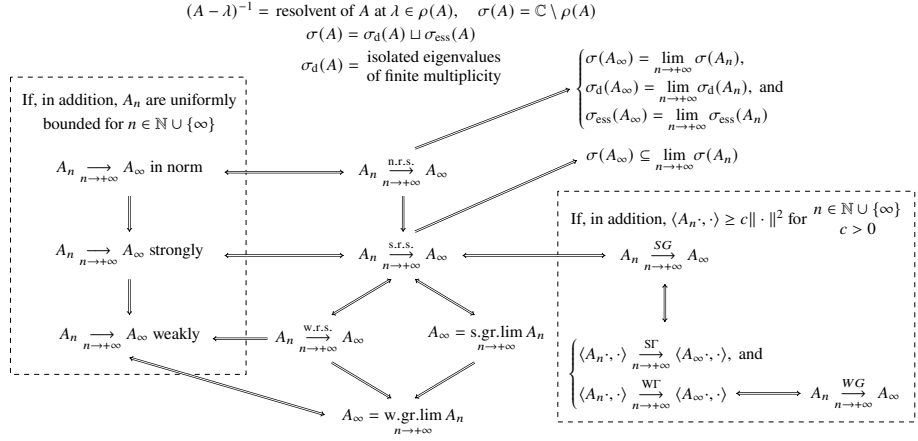

\section{Example: resolvent convergence of the Robin Laplacian} \label{sec:RobinLaplacian}

We conclude this chapter illustrating the resolvent convergence through an example. To this end, we first recall the basic notation regarding Sobolev spaces that we will use; see \cite[Chapter 9 and Section 11.4]{Brezis2011}, \cite[Chapter 5]{Evans2010}, or \cite[Chapter 4]{Taylor2011}.

In the sequel, $\Omega$ denotes a bounded domain in $\mathbb R^N$ with $C^2$ boundary, $N\geq 2$. We denote by $L^2(\Omega)$ the Hilbert space of functions $\varphi:\Omega\to\mathbb C$ endowed with the scalar product $\langle \cdot,\cdot \rangle_{L^2(\Omega)}$ and the associated norm $\|\cdot\|_{L^2(\Omega)}$, respectively defined as
\begin{equation*}
    \langle \varphi,\psi\rangle_{L^2(\Omega)}:=\int_{\Omega} \varphi \, \overline\psi \,dx \quad\text{and}\quad
    \|\varphi\|_{L^2(\Omega)}:=\langle \varphi,\varphi\rangle_{L^2(\Omega)}^{1/2}.
\end{equation*}
For $k\in \mathbb N$, we denote by $H^k(\Omega)$ the Sobolev space of functions in $L^2(\Omega)$ with weak partial derivatives up to order $k$ in $L^2(\Omega)$, and $H_0^k(\Omega)$ denotes the closure with respect to the $H^k(\Omega)$ norm of the set of smooth functions compactly supported in $\Omega$. 

Similarly, $L^2(\partial\Omega)$ denotes the Hilbert space of functions $\varphi:\partial\Omega\to\mathbb C$ endowed with the scalar product $\langle \cdot,\cdot \rangle_{L^2(\partial\Omega)}$ and the associated norm $\|\cdot\|_{L^2(\partial\Omega)}$, respectively defined as
\begin{equation*}
    \langle \varphi,\psi\rangle_{L^2(\partial\Omega)}:=\int_{\partial\Omega} \varphi \, \overline\psi \,d\upsigma \quad\text{and}\quad
    \|\varphi\|_{L^2(\partial\Omega)}:=\langle \varphi,\varphi\rangle_{L^2(\partial\Omega)}^{1/2}.
\end{equation*}
Here, $\partial\Omega$ denotes the boundary of $\Omega$ and $\upsigma$ denotes the surface measure on $\partial\Omega$. We denote by $H^{1/2}(\partial \Omega)$ the fractional Sobolev space of functions $\varphi\in L^2(\partial\Omega)$ such that 
\begin{equation*}
    \|\varphi\|_{H^{1/2}(\partial \Omega)}:= \Big( \int_{\partial\Omega} |\varphi|^2\,d\upsigma + \int_{\partial\Omega}\int_{\partial\Omega}\frac{|\varphi(x)-\varphi(y)|^2}{|x-y|^{2}} \,d\upsigma(y)\,d\upsigma(x) \Big)^{1/2} < +\infty.
\end{equation*}
We shall omit the measures of integration when no confusion arises.

The continuous dual space of $H^{1/2}(\partial \Omega)$ is denoted by $H^{-1/2}(\partial \Omega)$. The action of $\varphi \in H^{-1/2}(\partial \Omega)$ on $\psi \in H^{1/2}(\partial \Omega)$ is denoted by the pairing $\langle \varphi, \psi \rangle_{H^{-1/2}(\partial\Omega), H^{1/2}(\partial\Omega)}$, and the norm in $H^{-1/2}(\partial \Omega)$ is
\begin{equation*}
    \|\varphi\|_{H^{-1/2}(\partial \Omega)}:= 
        { \sup_{\|\psi\|_{H^{1/2}(\partial \Omega)}\leq1}} 
        \langle \varphi, \psi \rangle_{H^{-1/2}(\partial\Omega), H^{1/2}(\partial\Omega)}.
\end{equation*}
Recall that, whenever $\varphi \in L^2(\partial\Omega) \subset H^{-1/2}(\partial\Omega)$ and $\psi \in H^{1/2}(\partial\Omega) \subset L^2(\partial\Omega)$, the pairing satisfies
\begin{equation} \label{eq:Brezis}
    \langle \varphi, \psi \rangle_{H^{-1/2}(\partial\Omega), H^{1/2}(\partial\Omega)} = \overline{ \langle \varphi, \psi \rangle}_{L^2(\partial\Omega)};
\end{equation}
see, for example, \cite[Remark 3 in Section 5.2, and Section 11.4]{Brezis2011}. The reason why there is a complex conjugate in \eqref{eq:Brezis} is that we defined $\langle\cdot,\cdot\rangle_{L^2(\partial\Omega)}$ to be linear with respect to the first entry.

We want to study the convergence in a resolvent sense of the so-called \emph{Robin Laplacian}, defined for $a\in(0,+\infty)$ by
\begin{eqnarray*}
    \mathrm{Dom}(-\Delta_a) &:=& \{u\in H^2(\Omega): \partial_\nu u +au =0 \text{ in } H^{1/2}(\partial \Omega) \}, \\
    -\Delta_a u &:=& -\Delta u \quad \text{for every } u\in \mathrm{Dom}(-\Delta_a),
\end{eqnarray*}
as the boundary parameter $a$ moves in $(0,+\infty)$. Here, $\Delta$ denotes the Laplacian and $\nu$ denotes the outward unit normal vector to $\partial\Omega$. Heuristically, the operator formally obtained taking $a=0$ is the so-called \emph{Neumann Laplacian}, defined by
\begin{eqnarray*}
    \mathrm{Dom}(-\Delta_\mathrm{N}) &:=& \{u\in H^2(\Omega): \partial_\nu u =0 \text{ in } H^{1/2}(\partial \Omega) \}, \\
    -\Delta_\mathrm{N} u &:=& -\Delta u \quad \text{for every } u\in \mathrm{Dom}(-\Delta_\mathrm{N}).
\end{eqnarray*}
Instead, the operator formally obtained taking $a=+\infty$ is the so-called \emph{Dirichlet Laplacian}, defined by
\begin{eqnarray*}
    \mathrm{Dom}(-\Delta_\mathrm{D}) &:=& H^2(\Omega)\cap H^1_0(\Omega), \\
    -\Delta_\mathrm{D} u &:=& -\Delta u \quad \text{for every } u\in \mathrm{Dom}(-\Delta_\mathrm{D}).
\end{eqnarray*}
As we shall see later, the resolvent convergence formalizes these heuristics. 

It is well-known that the above operators are self-adjoint in $L^2(\Omega)$. In particular, by Remark~\ref{rmk:aboutDef}.\ref{item:resolventNoEmpty} the resolvent of $-\Delta_a$ at every $\lambda\in \mathbb C\setminus\mathbb R$ is a well defined bounded operator from $L^2(\Omega)$ to $L^2(\Omega)$, for every $a\in (0,+\infty)$. One might have the impression that working with the resolvent $(-\Delta_a-\lambda)^{-1}$ could be complicated, because it is the inverse of an operator. However, given $f\in L^2(\Omega)$, the resolvent $(-\Delta_a-\lambda)^{-1} f\in \mathrm{Dom}(-\Delta_a)$ is by definition the (unique) solution $\varphi_a$ to the elliptic boundary value problem
\begin{equation} \label{eq:SolutionRobin}
    \begin{cases}
        (-\Delta -\lambda) \varphi_a = f & \text{in } L^2(\Omega), \\
        \partial_\nu \varphi_a +a \varphi_a =0 & \text{in } H^{1/2}(\partial \Omega).
    \end{cases}
\end{equation}
Hence, applying the modern theory of elliptic PDEs ---see \cite[Chapter 9]{Brezis2011}, \cite[Chapter 6]{Evans2010} or \cite[Chapter 5]{Taylor2011}--- we can extract a lot of information from $(-\Delta_a-\lambda)^{-1} f$. The following result is an illustration of this, and shows that the resolvent $(-\Delta_a-\lambda)^{-1}$ is actually a bounded operator from $L^2(\Omega)$ to $H^1(\Omega)$.

\begin{proposition} \label{prop:boundedResolvent}
    For every $a>0$ and $\lambda\in \mathbb C\setminus \mathbb R$, there exists a constant $C_{\lambda}>0$ depending only on $\lambda$ such that 
    \begin{equation*}
        \|(-\Delta_a-\lambda)^{-1}f\|_{H^1(\Omega)} \leq C_{\lambda} \|f\|_{L^2(\Omega)}
    \end{equation*}
    for every $f\in L^2(\Omega)$. As a consequence, $(-\Delta_a-\lambda)^{-1}$ is a compact operator from $L^2(\Omega)$ to $H^1(\Omega)$, and $\sigma(-\Delta_a)$ is purely discrete.
\end{proposition}

\begin{proof}
    Fix $a>0$. Let $f\in L^2(\Omega)$, and set $\varphi_a := (-\Delta_a-\lambda)^{-1}f$. Since $\varphi_a$ solves the boundary value problem \eqref{eq:SolutionRobin}, by multiplying the PDE by $\overline{\varphi_a}$, integrating by parts, and using the boundary condition $\partial_\nu \varphi_a +a \varphi_a =0$ that $\varphi_a\in \mathrm{Dom}(-\Delta_a)$ satisfies, we get
    \begin{eqnarray*}
        \int_\Omega f \, \overline{\varphi_a} + \lambda \int_\Omega |\varphi_a|^2 &=& \int_\Omega -\Delta \varphi_a \, \overline{\varphi_a} = \int_\Omega |\nabla \varphi_a|^2-\int_{\partial\Omega} \partial_\nu \varphi_a \, \overline{\varphi_a} \\
        &=& \int_\Omega |\nabla \varphi_a|^2 + a \int_{\partial\Omega} |\varphi_a|^2.
    \end{eqnarray*}
    Hence, since $a>0$, by the triangle inequality, the Cauchy-Schwarz inequality, and the fact that 
    \begin{equation*}
        \|\varphi_a\|_{L^2(\Omega)} \leq \frac{1}{|\mathrm{Im}\lambda|} \|f\|_{L^2(\Omega)}
    \end{equation*}
    (see Remark~\ref{rmk:aboutDef}.\ref{eq:BoundResolvent}), we have
    \begin{eqnarray*}
        \|\nabla \varphi_a\|_{L^2(\Omega,\mathbb C^n)}^2 &\leq& \left| \int_\Omega f \, \overline{\varphi_a} + \lambda \int_\Omega |\varphi_a|^2 \right| \leq \|f\|_{L^2(\Omega)} \|\varphi_a\|_{L^2(\Omega)} + |\lambda| \|\varphi_a\|_{L^2(\Omega)}^2 \\
        &\leq& \left( \frac{1}{|\mathrm{Im}\lambda|} + \frac{|\lambda|}{|\mathrm{Im}\lambda|^2} \right)\|f\|_{L^2(\Omega)}^2.
    \end{eqnarray*}
    Therefore, 
    \begin{eqnarray*}
        \| \varphi_a\|_{H^1(\Omega)}^2 = \| \varphi_a\|_{L^2(\Omega)}^2 + \|\nabla \varphi_a\|_{L^2(\Omega,\mathbb C^n)}^2 \leq \frac{1+|\mathrm{Im}\lambda|+|\lambda|}{|\mathrm{Im}\lambda|^2} \|f\|_{L^2(\Omega)}^2,
    \end{eqnarray*}
    as desired. To conclude, since $H^1(\Omega)$ is compactly embedded in $L^2(\Omega)$ ---because $\Omega$ is bounded, see \cite[Theorem 9.16]{Brezis2011}---, by the former inequality we deduce that $(-\Delta_a-\lambda)^{-1}$ is a compact operator from $L^2(\Omega)$ to $L^2(\Omega)$. This, together with \cite[Proposition 8.8 in Appendix A and the paragraph below it]{Taylor2011}, yields the discreteness of $\sigma(-\Delta_a)$.
\end{proof}

We are now in the position to formalize the heuristic idea that the Robin Laplacian converges to the Neumann Laplacian as $a\to0^+$, and to the Dirichlet Laplacian as $a\to+\infty$. For the latter, we recall that the trace operator $t_{\partial\Omega}$ given by the trace theorem from $H^1(\Omega)$ to $H^{1/2}(\partial\Omega)$ \cite[Theorem 5.5.1]{Evans2010} extends into a bounded linear operator $t_\nu$ from $H(\mathrm{div},\Omega):=\{F\in L^2(\Omega,\mathbb C^n): \mathrm{div}F \in L^2(\Omega)\}$ to $H^{-1/2}(\partial\Omega)$ in the following sense: for all $F\in H^1(\Omega,\mathbb C^n)$, it holds that $t_\nu F = \nu \cdot t_{\partial\Omega} F$ almost everywhere on $\partial\Omega$. Moreover, 
\begin{equation} \label{eq:traceHdiv}
    \|\nu \cdot t_{\partial\Omega} F\|_{H^{-1/2}(\partial\Omega)} \leq C( \|F\|_{L^2(\Omega,\mathbb C^n)} + \|\mathrm{div} F\|_{L^2(\Omega)}),
\end{equation}
for some constant $C>0$ depending only on $\Omega$. 

\begin{proposition}
    The Robin Laplacian $-\Delta_a$ converges 
    \begin{enumerate}
        \item to the Neumann Laplacian $-\Delta_\mathrm{N}$, as $a\to0^+$, in the norm resolvent sense;
        \item to the Dirichlet Laplacian $-\Delta_\mathrm{D}$, as $a\to+\infty$, in the norm resolvent sense;
        \item to the Robin Laplacian $-\Delta_{a_0}$, as $a\to a_0$, in the norm resolvent sense, for every $a_0>0$.
    \end{enumerate}
\end{proposition}

\begin{proof}
    First, we prove the convergence as $a\to+\infty$. By Theorem~\ref{thm:RCenoughInOne}, it is enough to prove that the difference of resolvents at $\lambda=i$, given by
    \begin{equation*}
        W_{a,\mathrm{D}}:= (-\Delta_a-i)^{-1} - (-\Delta_\mathrm{D}-i)^{-1},
    \end{equation*}
    converges to zero in operator norm as $a\to+\infty$. To this end, let $f,g\in L^2(\Omega)$ and set
    \begin{eqnarray*}
        \varphi_a &:=& (-\Delta_a-i)^{-1}f \in \mathrm{Dom}(-\Delta_a) \subset H^2(\Omega), \\
        \psi &:=& (-\Delta_\mathrm{D}+i)^{-1} g \in \mathrm{Dom}(-\Delta_\mathrm{D}) \subset H^2(\Omega).
    \end{eqnarray*}
    By integrating by parts and by Remark~\ref{rmk:aboutDef}.\ref{item:rmkAdjointResolvent} applied to $-\Delta_\mathrm{D}$, we get
    \begin{equation} \label{eq:auxPairingD}
        \begin{split}
            \langle W_{a,\mathrm{D}} f, g \rangle_{L^2(\Omega)} &= \langle (-\Delta_a-i)^{-1}f - (-\Delta_\mathrm{D}-i)^{-1}f, g \rangle_{L^2(\Omega)} \\ 
        &= \langle (-\Delta_a-i)^{-1}f, g \rangle_{L^2(\Omega)} - \langle f, (-\Delta_\mathrm{D}+i)^{-1}g \rangle_{L^2(\Omega)} \\
        &= \langle \varphi_a, (-\Delta_\mathrm{D}+i)\psi \rangle_{L^2(\Omega)} - \langle (-\Delta_a-i)\varphi_a, \psi \rangle_{L^2(\Omega)} \\
        &= \langle \varphi_a, -\Delta \psi \rangle_{L^2(\Omega)} - \langle -\Delta \varphi_a, \psi \rangle_{L^2(\Omega)} \\
        &= - \langle \varphi_a, \partial_\nu \psi \rangle_{L^2(\partial\Omega)} + \langle \partial_\nu \varphi_a, \psi \rangle_{L^2(\partial\Omega)} \\
        &= \frac{1}{a} \langle \partial_\nu \varphi_a, \partial_\nu \psi \rangle_{L^2(\partial\Omega)},
        \end{split}
    \end{equation}
    where in the last equality we have used the boundary conditions $\partial_\nu \varphi_a + a \varphi_a=0$ and $\psi=0$, that $\varphi_a\in \mathrm{Dom}(-\Delta_a)$ and $\psi\in \mathrm{Dom}(-\Delta_\mathrm{D})$ satisfy. As we shall see afterward, and in view of \eqref{eq:auxPairingD}, in order to prove the convergence of $W_{a,\mathrm{D}}$ to zero in operator norm it is enough to bound the last pairing uniformly in $a$. This is what we do next.

    Since both $\partial_\nu \varphi_a$ and $\partial_\nu \psi$ belong to $H^{1/2}(\partial\Omega)$, by the property of the pairing \eqref{eq:Brezis}, the trace theorem from $H^1(\Omega)$ to $H^{1/2}(\partial\Omega)$ applied to $\nabla \psi$, and \eqref{eq:traceHdiv} applied to $F=\nabla\varphi_a$, we get
    \begin{equation} \label{eq:auxPairingBD}
        \begin{split}
            |\langle \partial_\nu \varphi_a, \partial_\nu \psi \rangle_{L^2(\partial\Omega)}| & \leq \|\partial_\nu \varphi_a\|_{H^{-1/2}(\partial\Omega)} \|\nabla \psi\|_{H^{1/2} (\partial\Omega)} \\
        & \leq C ( \|\nabla\varphi_a\|_{L^2(\Omega,\mathbb C^n)} + \|\Delta \varphi_a\|_{L^2(\Omega)} ) \|\psi\|_{H^2(\Omega)},
        \end{split}
    \end{equation}
    for some constant $C>0$ depending only on $\Omega$. On the one hand, since $\psi$ solves the boundary value problem
    \begin{equation*}
        \begin{cases}
            -\Delta\psi +i\psi = g & \text{in } L^2(\Omega), \\
            \psi =0 & \text{in } H^{1/2}(\partial \Omega),
        \end{cases}
    \end{equation*}
    with $g\in L^2(\Omega)$, by standard regularity theory ---see, for example, \cite[Theorem 4 in Section 6.3.2 and Remark in pg. 335]{Evans2010}--- we have
    \begin{equation} \label{eq:H2normD}
        \|\psi\|_{H^2(\Omega)} \leq C \|g\|_{L^2(\Omega)},
    \end{equation}
    for some constant $C>0$ depending only on $\Omega$. On the other hand, since $(-\Delta_a-i)\varphi_a=f$, by the triangle inequality and \eqref{eq:BoundResolvent} we have
    \begin{equation*}
        \|\Delta\varphi_a\|_{L^2(\Omega)} = \|f+i\varphi_a\|_{L^2(\Omega)} \leq \|f\|_{L^2(\Omega)} + \|\varphi_a\|_{L^2(\Omega)} \leq 2 \|f\|_{L^2(\Omega)},
    \end{equation*}
    Moreover, by Proposition~\ref{prop:boundedResolvent}, it holds that $\|\nabla\varphi_a\|_{L^2(\Omega,\mathbb C^n)}\leq 3 \|f\|_{L^2(\Omega)}$, and hence 
    \begin{equation} \label{eq:HdivnormPhia}
        \|\nabla\varphi_a\|_{L^2(\Omega,\mathbb C^n)} + \|\Delta \varphi_a\|_{L^2(\Omega)} \leq 5 \|f\|_{L^2(\Omega)}.
    \end{equation}
    Combining \eqref{eq:auxPairingBD}, \eqref{eq:H2normD} and \eqref{eq:HdivnormPhia} with \eqref{eq:auxPairingD}, we obtain
    \begin{equation*}
        |\langle W_{a,\mathrm{D}} f, g \rangle_{L^2(\Omega)}| \leq K/a \|f\|_{L^2(\Omega)} \|g\|_{L^2(\Omega)},
    \end{equation*}
    for some constant $K>0$ depending only on $\Omega$. Dividing both sides of this last inequality by $\|f\|_{L^2(\Omega)} \|g\|_{L^2(\Omega)}$, and taking the supremum among all functions $f,g\in L^2(\Omega)\setminus\{0\}$, we conclude that
    \begin{equation*}
        \|W_{a,\mathrm{D}}\|_{L^2(\Omega)\to L^2(\Omega)} \leq K/a.
    \end{equation*}
    Norm resolvent convergence follows by letting $a\to+\infty$. 
    
    The proof of the convergence to the Neumann Laplacian as $a\to0^+$ is analogous. As before, we will show that 
    \begin{equation*}
        W_{a,\mathrm{N}}:= (-\Delta_a-i)^{-1} - (-\Delta_\mathrm{N}-i)^{-1},
    \end{equation*}
    converges to zero in operator norm as $a\to0^+$. Given $f,g\in L^2(\Omega)$, set
    \begin{eqnarray*}
        \varphi_a &:=& (-\Delta_a-i)^{-1}f \in \mathrm{Dom}(-\Delta_a) \subset H^2(\Omega), \\
        \psi &:=& (-\Delta_\mathrm{N}+i)^{-1} g \in \mathrm{Dom}(-\Delta_\mathrm{N}) \subset H^2(\Omega).
    \end{eqnarray*}
    Similarly as in \eqref{eq:auxPairingD}, integration by parts and Remark~\ref{rmk:aboutDef}.\ref{item:rmkAdjointResolvent} applied to $-\Delta_\mathrm{N}$ straightforwardly lead to
    \begin{equation*}
        \langle W_{a,\mathrm{N}} f, g \rangle_{L^2(\Omega)} = -\langle \varphi_a, \partial_\nu \psi \rangle_{L^2(\partial\Omega)} +\langle \partial_\nu \varphi_a, \psi \rangle_{L^2(\partial\Omega)}.
    \end{equation*}
    Using now the boundary conditions $\partial_\nu \varphi_a + a \varphi_a = 0$ and $\partial_\nu \psi = 0$, that $\varphi_a\in \mathrm{Dom}(-\Delta_a)$ and $\psi\in \mathrm{Dom}(-\Delta_\mathrm{N})$ satisfy, we get that
    \begin{equation*}
        \langle W_{a,\mathrm{N}} f, g \rangle_{L^2(\Omega)} = -a\langle \varphi_a, \psi \rangle_{L^2(\partial\Omega)}.
    \end{equation*}
    As we did before, it is enough to bound the pairing in the right hand side uniformly in $a$. This is what we do next.

    Since both $\varphi_a$ and $\psi$ belong to $H^1(\Omega)$, by the trace theorem from $H^1(\Omega)$ to $H^{1/2}(\partial\Omega)$ we have
    \begin{eqnarray*}
        |\langle \varphi_a, \psi \rangle_{L^2(\partial\Omega)}| & \leq & \|\varphi_a\|_{L^2(\partial\Omega)} \|\psi\|_{L^2(\partial\Omega)} \leq \|\varphi_a\|_{H^{1/2}(\partial\Omega)} \|\psi\|_{H^{1/2}(\partial\Omega)} \\
            & \leq & C \|\varphi_a\|_{H^1(\Omega)} \|\psi\|_{H^1(\Omega)},
    \end{eqnarray*}
    for some constant $C>0$ depending only on $\Omega$. On the one hand, by Proposition~\ref{prop:boundedResolvent} there holds $\|\varphi_a\|_{H^1(\Omega)}\leq 3 \|f\|_{L^2(\Omega)}$. On the other hand, since $\psi$ solves the boundary value problem
    \begin{equation*}
        \begin{cases}
            -\Delta\psi = g-i\psi & \text{in } L^2(\Omega), \\
             \partial_\nu \psi =0 & \text{in } H^{1/2}(\partial \Omega),
        \end{cases}
    \end{equation*}
    then, by the triangle inequality, Cauchy-Schwarz, and \eqref{eq:BoundResolvent}, it holds that
    \begin{equation*}
        \|\psi\|_{H^1(\Omega)}^2 = \|\psi\|_{L^2(\Omega)}^2 + \|\nabla \psi\|_{L^2(\Omega,\mathbb C^n)}^2 \leq \|g\|_{L^2(\Omega)}^2 + \left| \int_\Omega (g-i\psi) \, \overline{\psi} \right|^2 \leq 5\|g\|_{L^2(\Omega)}^2. 
    \end{equation*}
    Combining the previous estimates, we obtain the desired uniform bound. 
    
    The norm resolvent convergence to $-\Delta_{a_0}$ as $a\to a_0$ readily follows. We sketch the proof omitting the details. Given $f,g\in L^2(\Omega)$, one sets 
    \begin{eqnarray*}
        \varphi_a &:=& (-\Delta_a-i)^{-1}f \in \mathrm{Dom}(-\Delta_a) \subset H^2(\Omega), \\
        \varphi_{a_0} &:=& (-\Delta_{a_0}+i)^{-1} g \in \mathrm{Dom}(-\Delta_{a_0}) \subset H^2(\Omega)
    \end{eqnarray*}
    and straightforwardly verifies that
    \begin{equation*}
        \langle (-\Delta_a-i)^{-1}f - (-\Delta_{a_0}-i)^{-1}f, g\rangle_{L^2(\Omega)} = (a_0-a) \langle\varphi_a, \varphi_{a_0} \rangle_{L^2(\partial\Omega)}.
    \end{equation*}
    Applying Proposition~\ref{prop:boundedResolvent} to both $\varphi_a$ and $\varphi_{a_0}$, one immediately bounds the pairing in the right hand side uniformly in $a$ and $a_0$, and concludes the proof of the theorem.
\end{proof}

It is clear from the previous proof that the same result remains true if instead of the Laplacian we consider uniformly elliptic operators in divergence form.

\begin{proposition}
    Let $A = (A_{ij}(x))_{i,j=1,\dots,n}\in C^{0,1}(\Omega, \mathbb C^n\times \mathbb C^n)$, and assume that $A$ is Hermitian, that is, such that $A_{ij}(x) = \overline{A_{ji}(x)}$ for almost every $x\in \Omega$. Furthermore, assume that there exist $0<\lambda<\Lambda<+\infty$ such that:
    \begin{enumerate}
        \item $|A(x)\xi|\leq \Lambda|\xi|$ for every $\xi\in \mathbb C^n$, for almost every $x\in \Omega$, and
        \item $A(x)\xi \cdot \overline \xi \geq \lambda |\xi|^2$ for every $\xi\in \mathbb C^n$, for almost every $x\in \Omega$.
    \end{enumerate}
    Then, the operator defined for $a\in(0,+\infty)$ as
    \begin{eqnarray*}
        \mathrm{Dom}(L_a) &:=& \{u\in H^2(\Omega): \nu\cdot A\nabla u +au =0 \text{ in } H^{1/2}(\partial \Omega) \}, \\
        L_a u &:=& -\mathrm{div}(A\nabla u) \quad \text{for every } u\in \mathrm{Dom}(L_a),
    \end{eqnarray*}
    \begin{enumerate}
        \item converges, as $a\to+\infty$, in the norm resolvent sense to the operator
        \begin{eqnarray*}
            \mathrm{Dom}(L_\mathrm{D}) &:=& H^2(\Omega)\cap H^1_0(\Omega), \\
            L_\mathrm{D} u &:=& -\mathrm{div}(A\nabla u) \quad \text{for every } u\in \mathrm{Dom}(L_\mathrm{D});
        \end{eqnarray*}
        \item converges, as $a\to0^+$, in the norm resolvent sense to the operator
        \begin{eqnarray*}
            \mathrm{Dom}(L_\mathrm{N}) &:=& \{u\in H^2(\Omega): \nu\cdot A\nabla u =0 \text{ in } H^{1/2}(\partial \Omega) \}, \\
            L_\mathrm{N} u &:=& -\mathrm{div}(A\nabla u) \quad \text{for every } u\in \mathrm{Dom}(L_\mathrm{N});
        \end{eqnarray*}
        \item converges, as $a\to a_0$, in the norm resolvent sense to $L_{a_0}$, for every $a_0>0$.
    \end{enumerate}
\end{proposition}

\begin{acknowledgement}
The author thanks Maicol Caponi, Alessandro Carbotti, Giuseppe Cosma Brusca, and Alberto Maione for enlightening discussions regarding the relations between resolvent and $G$- and $\Gamma$-convergence. The author also thanks the suggestions of the anonymous referee, which have enriched Section~\ref{sec:relationGamma}. 
\end{acknowledgement}

\ethics{Competing Interests}{ The author is supported by the Spanish grants PID2021-123903NB-I00 and RED2022-134784-T funded by MCIN/AEI/10.13039/501100011033, by ERDF ``A way of making Europe", and by the Catalan grant 2021-SGR-00087. This work is supported by the Spanish State Research Agency, through the Severo Ochoa and Mar\'ia de Maeztu Program for Centers and Units of Excellence in R\&D (CEX2020-001084-M), and more specifically by the grant CEX2020-001084-M-20-1. The author acknowledges CERCA Programme/Generalitat de Catalunya for institutional support.}

%
%
%

\end{document}